\documentclass[a4paper, 11pt]{article}
\usepackage{amsmath,amssymb,esint,amscd,xspace,fancyhdr,color,authblk,srcltx,fontenc,bbm}
\usepackage{hyperref}
\usepackage{cite}

\makeatletter
\newcommand{\leqnomode}{\tagsleft@true}
\newcommand{\reqnomode}{\tagsleft@false}
\makeatother

\newtheorem{theorem}{Theorem}[section]

\newtheorem{definition}[theorem]{Definition}

\newtheorem{lemma}[theorem]{Lemma}

\newtheorem{remark}[theorem]{Remark}

\newtheorem{taggedtheoremx}{Theorem}
\newenvironment{taggedtheorem}[1]
 {\renewcommand\thetaggedtheoremx{#1}\taggedtheoremx}
 {\endtaggedtheoremx}

\newenvironment{proof}[1][Proof]{\textbf{#1.} }{\hfill\rule{0.5em}{0.5em}}
{\catcode`\@=11\global\let\AddToReset=\@addtoreset
\AddToReset{equation}{section}

\AddToReset{theorem}{section}

\title{Lorentz estimates for quasi-linear elliptic double obstacle problems involving a Schr{\"o}dinger term}
\author{Thanh-Nhan Nguyen\thanks{Department of Mathematics, Ho Chi Minh City University of Education, Ho Chi Minh City, Vietnam; \texttt{nhannt@hcmue.edu.vn}}, Minh-Phuong Tran\footnote{Corresponding author.}\thanks{Applied Analysis Research Group, Faculty of Mathematics and Statistics, Ton Duc Thang University, Ho Chi Minh City, Vietnam; \texttt{tranminhphuong@tdtu.edu.vn}}}

\date{\today}

\begin{document}
\maketitle
\begin{abstract}

Our goal in this article is to study the global Lorentz estimates for gradient of weak solutions to $p$-Laplace double obstacle problems involving the Schr\"odinger term: $-\Delta_p u + \mathbb{V}|u|^{p-2}u$  with bound constraints $\psi_1 \le u \le \psi_2$ in non-smooth domains. This problem has its own interest in mathematics, engineering, physics and other branches of science. Our approach makes a novel connection between the study of Calder\'on-Zygmund theory for nonlinear Schr\"odinger type equations and variational inequalities for double obstacle problems.

\noindent 

\medskip


\medskip


\end{abstract}   
                  

\section{Introduction}
\label{sec:intro}
The calculus of variations is one of the classical and renowned topics in mathematical analysis, that has a wide range of interesting applications in many areas of physics, economics, engineering,  biology and so on. Most of problems in calculus of variations have origins in physics where one wishes to minimize (or maximize) the energy functionals subject to the given constraints. Apart from the studies on the existence and uniqueness of the solution to a variational problem in calculus of variations, the regularity (or smoothness) of such minimizers has also been the objective of intensive researches in recent years.

A problem in calculus of variations modeled with an inequality constraint leads to an obstacle problem, that can be characterized by a variational inequality. Theory of variational inequalities and free boundary problem, date back to the seminal works of Fichera,  Kinderlehrer and Stampacchia in \cite{Fichera, Stampacchia, KS1980}, was a classical topic that has attracted much attention in the last several years. The obstacle problem was motivated by many applications arising from physics, mechanics, engineering and other fields of applied sciences such as  membrane-fluid, fluid  filtration in porous media, investments with transaction costs in financial mathematics, elastic-plastic torsion, a game of `tug-of-war', etc. To the intimate connections, we recommend the reader to reference books in~\cite{Friedman,Troianiello,KS1980, Rodfrigues1987} for further mathematical problems and applications. The  prototype form of the obstacle problems is to find minimizers of integral energy functionals
\begin{align*}
\min_{u \ge \psi}{\int_\Omega{F(x,u(x),\nabla u(x)) dx}},
\end{align*}
in some domain $\Omega \subset \mathbb{R}^n$, where the unknown function $u$ is constrained to be greater than a given obstacle $\psi$. The study of obstacle problems is extended to several types of divergence form of elliptic operators and their suitable associated energy functionals. There have been a lot of works on the regularity of solutions for elliptic and parabolic variational inequalities with one obstacle constraint. In particular, in \cite{Choe1991, CL1991, Choe2016} authors proved the $C^{0,\alpha}$ and $C^{1,\alpha}$ regularity for quasilinear elliptic obstacle problems, H\"older continuity for minimizers of integral functionals under standard and non-standard growth in the works of Eleuteri \textit{et al.} \cite{Eleuteri2007, EH2011, EHL2013}, Calder\'on-Zygmund theory for elliptic/parabolic obstacle problems proposed in \cite{Choe2016,BDM2011,BCW2012}. Further, some mathematical tools have been developed to deal with obstacle with measure data (see \cite{Scheven1, Scheven2}), higher integrability (see \cite{BS2012}), Lorentz estimates (see \cite{Baroni2014}), etc and concerned papers are too many to cite, we only list here some of them for interesting readers.   

Along with the obstacle issues, the \emph{double obstacle problems} are also of interest. As seen with the word `double', problems can be further generalized with two obstacles: solution is constrained to lie between the lower and upper functions, i.e. $\psi_1 \le u \le \psi_2$. The line of research on double obstacle problems has also been developed in a rich literature. Devoted to regularity of solutions, we refer for instance to \cite{MMV1989} related to the double obstacle problems with linear operators and bounded measurable coefficients; problems involving degenerate elliptic operators in \cite{MZ1991, KZ1991,Lieberman1991}, or in \cite{BFM2001} for local $C^{1,\alpha}$ estimates with non-standard growth; and a lot of works treating the regularity estimates in certain spaces (see \cite{RT2011, BLO2020, BR2020}). As far as we are concerned, there seem to be fewer studies on double obstacle problems than single obstacle ones, even though it also arises in many applications. Hence, in the present paper, we will suggest an approach to prove the global regularity of solutions to double obstacle problems involving Schr\"odinger operators. 

More precisely, in this article, we consider the $p$-Laplace Schr\"odinger double obstacle problems, where the nonlinearity operator involves the $p$-Laplace $\Delta_p u = \mathrm{div}(|\nabla u|^{p-2}\nabla u)$ and a singular potential $\mathbb{V}$. Let $\Omega$ be an open bounded domain of $\mathbb{R}^n$ ($n \ge 2$);  $p \in (1,\infty)$ and $\mathbf{F} \in L^p(\Omega;\mathbb{R}^n)$. Given $\psi_1$, $\psi_2 \in W^{1,p}(\Omega)$ are two fixed functions such that $\psi_1 \le \psi_2$ almost everywhere in $\Omega$ and $\psi_1 \le 0 \le \psi_2$ on $\partial \Omega$, let us introduce the following convex admissible set related to $\psi_1$ and $\psi_2$ by
\begin{align}\label{def:S-0}
\mathbb{K}:= \left\{w \in W^{1,p}_0(\Omega): \ \psi_1 \le w \le \psi_2 \ \mbox{ a.e. in }\Omega\right\}.
\end{align}
We can summarize the form of such problem: to find the unknown function $u$ such that
\begin{align}
\label{eq:Lap_Schro}
- \Delta_p u + \mathbb{V}|u|^{p-2}u \le - \mathrm{div}(|\mathbf{F}|^{p-2}\mathbf{F}), \quad \text{a.e. in} \ \mathbb{K}.
\end{align}
Naturally, in the sense of calculus of variations, this problem related to the minimizers of the constrained functional
\begin{align}
\label{eq:min}
w \mapsto \frac{1}{p}\int_{\Omega}{\left( |\nabla w|^p + \mathbb{V}|w|^p \right) dx} - \int_\Omega{\langle |\mathbf{F}|^{p-2}\mathbf{F},\nabla w \rangle dx}, \ \mbox{ in } \ \mathbb{K}.
\end{align}

In the case when $\mathbb{V}\equiv 0$, one simply has the double obstacle problem for $p$-Laplacian. A plenty of regularity results are obtained in this respect, Choe in \cite{Choe1992} obtained $C^{0,\alpha}$ and $C^{1,\alpha}$-estimates for solutions and a extensive literature on regularity for solutions/minima to a class of variational integrals of more general types. For instance, for nonlinear elliptic double obstacle problems, the interior H\"older estimate for problems nonstandard growth studied in \cite{BFM2001}, global $L^q$ estimates by Byun \textit{et al.} in \cite{BR2020,BLO2020}. Associated with nonlinear elliptic equations (without obstacles), there have been various results pertaining to regularity theory and Calder\'on-Zygmund, together with some technical issues discussed in \cite{AM2007,Duzamin2,Mi3, CoMi2016, BW2, CP1998, 55QH4, MPT2018, MP12, PNJDE} and the further references to be continued.

In the presence of an appropriate potential $\mathbb{V}(x)$, we have the time independent Schr\"odinger problems of $p$-Laplace type.  \emph{Without obstacles}, equation \eqref{eq:Lap_Schro} is one of the most important research issues in classical physics, that has been broadly studied in recent decades. Arising in different physical contexts, elliptic equations involving Schr\"odinger operators have wide applications, specifically in quantum mechanics, non-Newtonian fluid theory and gas flow in porous media, etc, (see \cite{BS1991}). There is extensive literature on the study of gradient estimates (regularity) for linear/nonlinear elliptic Schr\"odinger type equations, we refer to the recent works \cite{AMS2001,Sugano,Liu,Shen,Shen2,BBHV,LO_schro} where it is possible to find further references therein. For instance, for $p=2$, the $L^q$-estimates and Calder\'on-Zygmund type estimates were established by Shen in a fine paper \cite{Shen} when the nonnegative potential $\mathbb{V}$ belongs to the reverse H\"older class $\mathcal{RH}^{\gamma}$, $\gamma >\frac{n}{p}$. An additive interesting result was later obtained by Sugano in \cite{Sugano} when potential $\mathbb{V}$ moreover includes non-negative polynomials. The extension of such results were provided for non-divergence form linear elliptic equations with VMO coefficients in \cite{BBHV}, and for nonlinear divergence elliptic equations recently established by Lee and Ok in \cite{LO_schro}, they further states global $L^q$ estimates in terms of Calder\'on-Zygmund in the same paper.

In the present paper, we are interested in the study of double obstacle problems for quasilinear elliptic equations involving Schr\"odinger operators. On the whole, we extend the results of \cite{LO_schro} from variational equations to variational inequalities with constraints, and moreover, the global regularity estimates of the weak solutions are obtained in the setting of Lorentz spaces.

Let us here describe the way of formulating the double obstacle problems discussed in our work. Here,  we are interested in a class of quasilinear elliptic double obstacle problems, which is more general than that of  \eqref{eq:Lap_Schro}.  We actually consider the problems involving Schr\"odinger term of the type
\begin{align}\label{eq:quasi}
-\mathrm{div}(\mathcal{A}(x,\nabla u)) + \mathbb{V} |u|^{p-2} u \le  f -\mathrm{div}(\mathcal{B}(x,\mathbf{F})), \quad \text{a.e. in}\ \ \mathbb{K}.
\end{align}
Here, $\mathcal{A}, \mathcal{B}: \ \Omega \times \mathbb{R}^n \rightarrow \mathbb{R}$ are quasi-linear Carath\'eodory vector valued operators satisfying the natural conditions: there exist constants $p \in (1,\infty)$ and $\Lambda>0$ such that 
\begin{align}\label{eq:A1-DOP}
& \hspace{1cm} \left| \mathcal{A}(x,\mu) \right| + \left|\langle D_{\mu} \mathcal{A}(x,\mu), \mu\rangle \right| +  \left| \mathcal{B}(x,\mu) \right|  \le \Lambda |\mu|^{p-1}, \\ \label{eq:A2-DOP}
& \langle \mathcal{A}(x,\mu_1)-\mathcal{A}(x,\mu_2), \mu_1 - \mu_2 \rangle \ge \Lambda^{-1} \left(|\mu_1|^2 + |\mu_2|^2 \right)^{\frac{p-2}{2}}|\mu_1 - \mu_2|^2,
\end{align}
for almost every $x$ in $\Omega$ and every $\mu$, $\mu_1$, $\mu_2 \in \mathbb{R}^n \setminus \{0\}$. 

Before formulating the main results, let us introduce several important and relevant terminologies. Throughout this paper, we consider the potential $\mathbb{V}$ which belongs to the reverse H{\"o}lder class $\mathcal{RH}^{\gamma}$, for some $\gamma \in [\frac{n}{p},n)$. Moreover we require an extra assumption on $\mathbb{V}$, that $\|\mathbb{V}\|_{L^{\gamma;p\gamma}(\Omega)} \le 1$ for better gradient estimates (see Section \ref{sec:pre} for detailed explanation). Further, it is also remarkable that the given data $\mathbf{F} \in L^p(\Omega, \mathbb{R}^n)$, $f \in L^{p'}(\Omega)$ ($p'$ stands for the exponent conjugate to $p$), and $\psi_1, \psi_2 \in W^{1,p}(\Omega)$ for $p \in (1,n)$. For simplicity, we shall denote  
$$|\mathcal{F}|^p := |f|^{p'} + |\mathbf{F}|^{p} +  \mathbb{E}(\psi_1) + \mathbb{E}(\psi_2),$$ 
where the function $\mathbb{E}: \ W^{1,p}(\Omega) \to [0,\infty)$ defined by $\mathbb{E}(v) = |\nabla v|^p + \mathbb{V}|v|^p$. It is worth pointing out that $\mathbb{V}|v|^p \in L^1(\Omega)$ for every $v \in W^{1,p}(\Omega)$ and $\mathbb{V} \in \mathcal{RH}^{\gamma}$ by H{\"o}lder's and Sobolev's inequalities. 

In the present paper, we study the weak solutions to double obstacle problem \eqref{eq:quasi}, i.e. solutions of the weak formulation of variational inequality. The natural notion of weak solutions is given as follows.
\begin{definition}\label{def:weak_sol}
Let $\mathbf{F} \in L^p(\Omega;\mathbb{R}^n)$ and $f \in L^{p'}(\Omega)$ for $p>1$. We say that a function $u \in \mathbb{K}$ is a weak solution to double obstacle problem~\eqref{eq:quasi} if the following variational inequality
\begin{align}\nonumber
\int_{\Omega} \langle \mathcal{A}(x,\nabla u), \nabla (u - \varphi) \rangle dx & + { \int_{\Omega} \langle \mathbb{V}|u|^{p-2}u, u - \varphi\rangle  dx} \\ \tag{$\mathbf{P}$}\label{eq:DOP}
 & \qquad \le { \int_{\Omega} f(u - \varphi)  dx} + \int_{\Omega} \langle \mathcal{B}(x, \mathbf{F}), \nabla (u - \varphi) \rangle dx, 
\end{align}
holds for all test functions $\varphi \in \mathbb{K}$.
\end{definition}

It is worth mentioning that this paper does not focus on the existence of a solution, but as a consequence of theory for monotonicity and coercive properties (from \cite[Chapter 4]{KS1980} and \cite{RT2011,BBHV}), it enables us to establish the existence of solutions to our problem. Moreover, if $f=0$ and $\mathbb{V}\equiv 0$, there exists a positive constant $C$ such that the following estimate holds
\begin{align*}
\|\nabla u\|_{L^p(\Omega)} \le C \left(\|\mathbf{F}\|_{L^p(\Omega)} + \|\nabla \psi_1\|_{L^p(\Omega)} + \|\nabla \psi_2\|_{L^p(\Omega)}\right).
\end{align*}

Before discussing the technique used in this paper, let us review some studies in Calder\'on-Zygmund theory regarding  the nonlinear elliptic and parabolic partial differential equations. Originating in a beautiful approach by Caffarelli and Peral \cite{CP1998}, to obtain local $W^{1,p}$ estimates for a class of $p$-Laplace equations (relies on the Calder\'on-Zygmund decomposition and the boundedness of Hardy-Littlewood maximal functions), many far reaching results have been growing  since then. Also here, we would like to mention an outstanding work, shown in a paper \cite{AM2007} by Acerbi and Mingione. They presented an effective technique to give a proof of Calder\'on-Zygmund estimates with no use of harmonic analysis and maximal operators.  Afterwards, this approach has yielded a multitude of beautiful results in regularity theory. A different approach for constructing $L^q$ estimate for higher order elliptic and parabolic systems (divergence and non-divergence types) was proposed by Dong, Kim and Krylov in \cite{DK2011,Krylov}. This approach based on estimates involving sharp and maximal functions using a version of the Fefferman-Stein theorem, is also a successful tool to deal with Calder\'on-Zygmund. Moreover, it is also important for us here to point out the geometrical approach firstly presented by Byun and Wang in \cite{BW2}. This is a unifying method to prove the interior and boundary estimates for weak solutions of a class of general elliptic/parabolic equations, valuable and has been successfully applied in many regularity results. Our technique is somewhat the improvement of a surprising approach introduced in \cite{Mi3,AM2007,BW2} dealt with the Calder\'on-Zygmund estimates for quasilinear elliptic/parabolic systems. The approach then inspired a lot of other research on the regularity theory of certain partial differential equations (see \cite{55QH4,MPT2018,PNCRM,PNJDE,PNnonuniform,PNmix}). Motivated by such effective method, in our proposed technique, global gradient estimates of solutions to double obstacle problems \eqref{eq:DOP} are preserved under fractional maximal operators. We also turn the reader's attention to \cite{PN_dist} for another viewpoint of this approach, that will take advantage of \emph{fractional maximal distribution functions} in our efforts.

Let us now state our main results via the following theorems. In this work, we always assume $\Omega$ the Reifenberg flat domain (such that $\partial\Omega$ is sufficiently flat) and the nonlinearity $\mathcal{A}$ further has small BMO seminorm. For the sake of brevity, with given $\delta, r_0>0$, the assumption on $\Omega$ - a $(r_0,\delta)$-Reifenberg flat domain  together with operator $\mathcal{A}$ satisfying $[\mathcal{A}]^{r_0} \le \delta$, which will henceforth be called \emph{the assumption $(r_0,\delta)-(\mathbb{H})$} throughout the paper. The precise description of these assumptions  will be given in Section \ref{sec:pre}. Under the given assumptions on the operator $\mathcal{A}$ and $\mathcal{B}$ in~\eqref{eq:A1-DOP}-\eqref{eq:A2-DOP}, for the shortness of notation, let us use the word ``\texttt{data}'' is the set of parameters $\texttt{data} = \left(n,p,\Lambda,\gamma,\mathrm{diam}(\Omega)\right)$, in order to illustrate the dependence on given data of the problem considered.

\begin{taggedtheorem}{A}[Level-set inequality of measuring sets] 
\label{theo:main-A}
For every $\theta \in (\frac{1}{\gamma},1)$, one can find some constants $\varepsilon_0 = \varepsilon_0(\texttt{data},\theta) \in (0,1)$, $\delta = \delta(\texttt{data},\theta,\varepsilon)>0$ and $\beta = \beta(\texttt{data},\theta,\varepsilon) > 0$ such that if problem~\eqref{eq:DOP} satisfies the assumption $(r_0,\delta)-(\mathbb{H})$ for some $r_0>0$, then the following estimate
\begin{align}\label{eq:mainlambda-P} 
& \mathcal{L}^n(\{{\mathbf{M}}(\mathbb{E}(u))>\varepsilon^{-\theta}\lambda, {\mathbf{M}}(|\mathcal{F}|^p) \le \beta \lambda\}) \leq C \varepsilon \mathcal{L}^n(\{ {\mathbf{M}}(\mathbb{E}(u))> \lambda\}),
\end{align}
holds for any $\varepsilon \in (0,\varepsilon_0)$ and $\lambda>0$, where $C = C(\texttt{data}, \theta, \mathrm{diam}(\Omega)/r_0)>0$.
\end{taggedtheorem}

Here, for simplicity, we use the notation $\mathcal{L}^n(E)$ for Lebesgue measure of $E \subset \mathbb{R}^n$ and as a minor abuse of notation in Theorem \ref{theo:main-A} and what follows, we write $\{ |g| \ge \lambda\}$ instead of $\{x \in \Omega: |g(x)| > \lambda \}$. As a direct outcome of Theorem \ref{theo:main-A}, Theorem \ref{theo:main-norm} establishes the gradient estimates of solutions to our problem \eqref{eq:DOP}.

\begin{taggedtheorem}{B}[Global Lorentz regularity]
\label{theo:main-norm}
For every $0< q < \gamma$ and $0<s \le \infty$, there exists $\delta = \delta(\texttt{data},q,s)>0$ such that if problem~\eqref{eq:DOP} satisfies the assumption $(r_0,\delta)-(\mathbb{H})$ for some $r_0>0$ then
$$|\mathcal{F}|^p \in L^{q,s}(\Omega) \ \Longrightarrow \ \mathbb{E}(u) \in L^{q,s}(\Omega).$$ 
More precisely, there exists a positive constant $C = C(\texttt{data}, q, s)$ such that 
\begin{align}\label{eq:main-A}
\|\mathbb{E}(u)\|_{L^{q,s}(\Omega)} & \le  C  \||\mathcal{F}|^p\|_{L^{q,s}(\Omega)}.
\end{align}
\end{taggedtheorem}

Even though in Theorem \ref{theo:main-norm}, we state the norm estimates for gradient of weak solutions to our problem in the setting of Lorentz spaces, we nevertheless remark that the arguments can still be refined with fractional maximal operator $\mathbf{M}_\alpha$, to achieve a more general result. Let us state such result in the following Theorem \ref{theo:main-M-alpha}.

\begin{taggedtheorem}{C}[Lorentz regularity via fractional maximal operators]
\label{theo:main-M-alpha}
For every 
$$0 \le \alpha < \frac{n}{\gamma}, \quad 0< q < \frac{n\gamma}{n-\alpha\gamma}, \ \mbox{ and } \ 0<s \le \infty,$$ 
there exists $\delta = \delta(\texttt{data},\alpha,q,s)>0$ such that if problem~\eqref{eq:DOP} satisfies the assumption $(r_0,\delta)-(\mathbb{H})$ for some $r_0>0$ then
$$\mathbf{M}_{\alpha}(|\mathcal{F}|^p) \in L^{q,s}(\Omega) \ \Longrightarrow \ \mathbf{M}_{\alpha}(\mathbb{E}(u)) \in L^{q,s}(\Omega).$$ 
More precisely, there exists a positive constant $C = C(\texttt{data}, \alpha, q, s)$ such that 
\begin{align}\label{eq:main-C}
\|\mathbf{M}_{\alpha}(\mathbb{E}(u))\|_{L^{q,s}(\Omega)} & \le  C  \|\mathbf{M}_{\alpha}(|\mathcal{F}|^p)\|_{L^{q,s}(\Omega)}.
\end{align}
\end{taggedtheorem}

The organization of the paper contents is resumed as follows. Section \ref{sec:pre} is devoted to the notation, definitions and a few preliminary results that will be needed to prove our main results of the paper. Next in section \ref{sec:comparison} we establish a series of comparison estimates between the solutions $u$ to \eqref{eq:DOP} and some suitable homogeneous problems. The step of proving these comparisons is the key ingredient and most important to obtain our regularity results. Finally, by following the idea of `good-$\lambda$' technique, in the last section we are able to prove our main theorems, Theorem \ref{theo:main-A} and \ref{theo:main-norm}.

\section{Preliminaries}\label{sec:pre}

Let us in this section provide some preliminaries and prove preparatory results that will be used in the proofs of main theorems. 

We first introduce much of the notation and some basic definitions that will be used in the whole paper. In what follows, $\Omega \subset \mathbb{R}^n$ is always assumed to be an open bounded domain, and an open ball centered at $x_0$ with radius $\varrho>0$ in $\mathbb{R}^n$ is denoted by $B_{\varrho}(x_0)$. Moreover, the integral mean value of a function $g \in L_{loc}^1(\Omega)$ over a set ${B} \subset \mathbb{R}^n$ will defined by
\begin{align*}
\fint_{B}{g(x)dx} = \frac{1}{\mathcal{L}^n({B})}{\int_{B}{g(x)dx}}.
\end{align*}
In the entirety of the paper we will use $\mathrm{diam}(E)$ to mention the diameter of a set $E \subset \Omega$. 

By the symbol $C$, we denote a universal constant (larger than or equal to 1) whose exact value is unimportant and may vary from line to line. The dependencies of $C$ on some prescribed parameters will be highlighted between parentheses, if needed. For instance, when writing $C(\texttt{data})$ we mean that constant $C$ depends only on the given \texttt{data}.

\begin{definition}[$(r_0,\delta)$-Reifenberg flat domain]\label{def:Reifenberg}
For $0 < \delta < 1$ and $r_0>0$, we say that $\Omega$ is a $(r_0,\delta)$-Reifenberg flat domain if for each $\xi \in \partial \Omega$ and each $\varrho \in (0,r_0]$, it is possible to find a coordinate system $\{y_1,y_2,...,y_n\}$ with origin at $\xi$ such that
\begin{align*}
B_{\varrho}(\xi) \cap \{y_n > \delta \varrho\} \subset B_{\varrho}(\xi) \cap \Omega \subset B_{\varrho}(\xi) \cap \{y_n > -\delta \varrho\}.
\end{align*}
\end{definition}

\begin{definition}[$(r_0,\delta)$-BMO condition]\label{def:BMOcond}
Let $\delta >0$ and $r_0>0$, the operator $\mathcal{A}$ is said that satisfying a $(r_0,\delta)-\mathrm{BMO}$ condition if
\begin{align}\label{cond:BMO}
[\mathcal{A}]^{r_0} = \sup_{y \in \mathbb{R}^n, \ 0<\varrho\le r_0} \left(\fint_{B_{\varrho}(y)} \left(\sup_{\mu \in \mathbb{R}^n \setminus \{0\}} \frac{|\mathcal{A}(x,\mu) - \overline{\mathcal{A}}_{B_{\varrho}(y)}(\mu)|}{|\mu|^{p-1}}\right) dx\right) \le \delta,
\end{align}
where $\overline{\mathcal{A}}_{B_{\varrho}(y)}(\mu)$ denotes the average of $\mathcal{A}(\cdot,\mu)$ over the ball $B_{\varrho}(y)$.
\end{definition}

\begin{definition}[Reverse H{\"o}lder class]
Let $\mathbb{V} \in L^1_{\mathrm{loc}}(\mathbb{R}^n)$  be a non-negative function. We say that potential $\mathbb{V}$ belongs to reverse H{\"o}lder class $\mathcal{RH}^{\gamma}$ for some $1<\gamma \le \infty$ if there exists $C = C(n,\gamma)>0$ such that
\begin{align*}
\begin{cases}
\displaystyle{\left(\fint_{B_{\varrho}(x)} \mathbb{V}(\xi)^{\gamma}d\xi\right)^{\frac{1}{\gamma}}} &\le C \displaystyle{\fint_{B_{\varrho}(x)} \mathbb{V}(\xi)d\xi}, \quad \text{if} \ \ \gamma \in (1,\infty),\\[5pt]
\qquad \displaystyle{\|\mathbb{V}\|_{L^\infty(B_{\varrho}(x))}} &\le C \displaystyle{\fint_{B_{\varrho}(x)}{\mathbb{V}(\xi) d\xi}}, \quad \text{if} \ \ \gamma=\infty,
\end{cases}
\end{align*}
holds for every $x \in \mathbb{R}^n$ and $\varrho>0$.
\end{definition}

\begin{remark}
In some particular cases of considerable physical context, the potential $\mathbb{V}$ has its own significance. In many partial differential equations arising in engineering and physics, there are many discussions on such potential $\mathbb{V}$ such as octic potential, decatic potential, polynomial potential, etc. Here for typical examples of $\mathbb{V}$ in reverse H\"older class, we particular have $\mathbb{V}(x) = |x|^{-\alpha} \in \mathcal{RH}^\gamma$ for $\alpha<\frac{n}{\gamma}$ or when $\gamma=\infty$, positive polynomials belong to $\mathcal{RH}^\gamma$.
\end{remark}
\begin{remark}
If $\mathbb{V} \in \mathcal{RH}^\gamma$, then $\mathbb{V} \in \mathcal{RH}^{\gamma+\varepsilon}$ for some $\varepsilon>0$; and for all $0<\gamma<\infty$, $\mathcal{RH}^\gamma \subset \mathcal{RH}^\infty$.
\end{remark}
\begin{remark}
According to the Calder\'on-Zygmund type estimates proved in \cite[Theorem 2.3]{LO_schro}, one further suitable assumption on the potential $\mathbb{V}$ will be added to our problem, that related to the Morrey norm of $\mathbb{V}$:
\begin{align*}
\|\mathbb{V}\|_{L^{\gamma; p\gamma}(\Omega)} = \sup_{0<\varrho< \mathrm{diam}(\Omega); \, \xi \in \Omega}{\varrho^{p-\frac{n}{\gamma}}}\|\mathbb{V}\|_{L^{\gamma}(B_{\varrho}(\xi)\cap\Omega)}  \le 1.
\end{align*}
Based on the analysis performed in Lebesgue setting by Lee and Ok in \cite{LO_schro}, this appropriate condition was added in their proof due to the Reifenberg flatness assumed on domain $\Omega$ and in order to complete regularity results in the setting of Lorentz, this condition is unavoidable to assume. 
\end{remark}

As mentioned above, before passing to our main regularity results in the framework of Lorentz space, let us include here its definition. Lorentz space is one of important generalizations of the weak and classical $L^q$ spaces, that is affected by two scale parameters $q$ and $s$. We briefly recall it as follows. 
\begin{definition}[Lorentz spaces]
Let $0<q<\infty$ and $0<s\le \infty$, the Lorentz space, adopted with notation $L^{q,s}(\Omega)$, is the set all of function $g \in L^1_{\mathrm{loc}}(\Omega)$ such that $\|g\|_{L^{q,s}(\Omega)}$ is finite, where
\begin{align*}
\|g\|_{L^{q,s}(\Omega)} := \begin{cases} \displaystyle{\left[ q \int_0^\infty{ \lambda^s \mathcal{L}^n(\{\xi \in \Omega: |g(\xi)|>\lambda\} )^{\frac{s}{q}} \frac{d\lambda}{\lambda}} \right]^{\frac{1}{s}}}, & \ \mbox{ if } \ s < \infty,\\[8pt]
 \displaystyle{\sup_{\lambda>0}{\lambda \mathcal{L}^n(\{\xi \in \Omega:|g(\xi)|>\lambda\})^{\frac{1}{q}}}}, & \ \mbox{ if } \ s = \infty.\end{cases}
\end{align*}
\end{definition}
Next, we recall in the following the definition of fractional maximal operators and \textit{Hardy-Littlewood maximal operator}.
\begin{definition}[Maximal operators]\label{def:Malpha}
For $\alpha \in [0, n]$, the fractional maximal operator $\mathbf{M}_{\alpha}$ of a mapping $g \in L^1_{\mathrm{loc}}(\mathbb{R}^n)$ is defined by:
\begin{align} \nonumber
\mathbf{M}_\alpha g(\xi) = \sup_{\rho>0}{{\rho}^{\alpha} \fint_{B_{\rho}(\xi)}{|g(x)|dx}}, 
\end{align}
for every $\xi \in \mathbb{R}^n$. In particular, $\mathbf{M}_0$ is not different to the well-known Hardy-Littlewood maximal operator $\mathbf{M}$ given by:
\begin{align}\nonumber 
\mathbf{M}g(\xi) = \sup_{\rho >0}{\fint_{B_{\rho}(\xi)}|g(x)|dx}, \quad \xi \in \mathbb{R}^n, \ g \in L^1_{\mathrm{loc}}(\mathbb{R}^n).
\end{align}
\end{definition}
Further, from the above definitions of $\mathbf{M}$ and $\mathbf{M}_\alpha$, we have a nice property on the boundedness of fractional maximal function, provided in the next lemma.
\begin{lemma}[See~\cite{PNnonuniform}]\label{lem:bound-M}
Let $s \in [1,\infty)$ and $\alpha \in \left[0,\frac{n}{s}\right)$, there is $C>0$ such that 
\begin{align*}
\mathcal{L}^n\left(\left\{x \in \mathbb{R}^n: \ \mathbf{M}_{\alpha}g(x)> t \right\}\right) \le C \left(\frac{1}{t^{s}}\int_{\mathbb{R}^n}|g(x)|^s dx\right)^{\frac{n}{n-\alpha s}},
\end{align*}
for all $t >0$ and $g \in L^s(\mathbb{R}^n)$.
\end{lemma}

\section{Comparison results for double obstacle problems}
\label{sec:comparison}

We first introduce a function $\mathbf{\Phi}: \ \mathbb{R}^n \times \mathbb{R}^n \to [0,\infty)$ defined by
\begin{align}\label{def:func-Phi}
\mathbf{\Phi}(x, z) : = \left(|x|^2 + |z|^2 \right)^{\frac{p-2}{2}}|x - z|^2, \quad x, \, z \in \mathbb{R}^n.
\end{align} 
In fact, this function is considered to simplify the condition of $\mathcal{A}$ in~\eqref{eq:A2-DOP} as follows
\begin{align*}
\langle \mathcal{A}(x,\mu_1)-\mathcal{A}(x,\mu_2), \mu_1 - \mu_2 \rangle \ge \Lambda^{-1} \mathbf{\Phi}(\mu_1,\mu_2),
\end{align*}
for almost every $x$ in $\Omega$ and every $\mu$, $\mu_1$, $\mu_2 \in \mathbb{R}^n \setminus \{0\}$. Let us begin this section with a technical lemma, that will be stated and proved hereafter. 
\begin{lemma}\label{lem:Phi(v,v)}
Let $B \subset \mathbb{R}^n$ be an open bounded set and two given mappings $g_1$, $g_2 \in W^{1,p}(B)$ with $p>1$. Then for each $\varepsilon \in (0,1)$, one can find $C>0$ depending on $\varepsilon$ such that
\begin{align}\label{eq:lem-Phi(u,v)}
\fint_{B} \mathbb{E}(g_1 - g_2) dx \le \varepsilon  \fint_{B} \mathbb{E}(g_1) dx + C \fint_{B} \mathbf{\Phi}(\nabla g_1, \nabla g_2) + \mathbb{V}\mathbf{\Phi}(g_1, g_2) dx.
\end{align}
\end{lemma}
\begin{proof}
Inequality~\eqref{eq:lem-Phi(u,v)} is obviously when $p \ge 2$.  Otherwise, if $ p \in (1, 2)$, we first decompose 
\begin{align*}
\mathbb{V}|g_1 - g_2|^p & \le 2^{\frac{p(2-p)}{4}}\left[\left(\mathbb{V}|g_1|^p + \mathbb{V}|g_1-g_2|^p\right)\right]^{1-\frac{p}{2}} [\mathbb{V}\mathbf{\Phi}(g_1, g_2)]^{\frac{p}{2}},
\end{align*}
and then apply H{\"o}lder and Young inequalities for any $\epsilon \in (0,1/2)$, to get that
\begin{align}\nonumber
\fint_{B} \mathbb{V}|g_1 - g_2|^p dx & \le \epsilon  \fint_{B} \left(\mathbb{V}|g_1|^p + \mathbb{V}|g_1-g_2|^p\right) dx + 2\epsilon^{1 - \frac{2}{p}} \fint_{B} \mathbb{V} \mathbf{\Phi}(g_1, g_2) dx.
\end{align}
Using the same method for the gradient term and setting $\varepsilon = 2\epsilon$, it allows us to conclude~\eqref{eq:lem-Phi(u,v)}.
\end{proof}

This section is dedicated to state and prove some comparison estimates between weak solutions of our double obstacle problem \eqref{eq:DOP} and some homogeneous equations, via the following important lemmas. 

\begin{lemma}\label{lem:weak-max}
Suppose that $w_1, w_2 \in W^{1,p}(\Omega)$ for some $p>1$ satisfying $(w_1 - w_2)^+ \in W_0^{1,p}(\Omega)$, $w_1 \le w_2$ on $\partial \Omega$ and the following variational inequality
\begin{align}\nonumber
\int_{\Omega}  \left\langle\mathcal{A}(x,\nabla w_1), \nabla \varphi\right\rangle dx & + {  \int_{\Omega} \langle \mathbb{V}|w_1|^{p-2}w_1, \varphi \rangle dx} \\ \label{est:weak}
&  \le \int_{\Omega} \left\langle\mathcal{A}(x,\nabla w_2), \nabla \varphi\right\rangle dx + {  \int_{\Omega} \langle \mathbb{V}|w_2|^{p-2}w_2, \varphi \rangle dx},  
\end{align}
holds for all non-negative $\varphi \in W_0^{1,p}(\Omega)$. Then $w_1 \le w_2$ almost everywhere in $\Omega$.
\end{lemma}
\begin{proof}
By choosing $\varphi = (w_1 - w_2)^+ \in W_0^{1,p}(\Omega)$ in~\eqref{est:weak}, one has
\begin{align*}
\int_{\Omega} & \left\langle\mathcal{A}(x,\nabla w_1)-\mathcal{A}(x,\nabla w_2), \nabla ((w_1 - w_2)^+)\right\rangle dx \\
& \hspace{3cm} + {  \int_{\Omega} \langle \mathbb{V}|w_1|^{p-2}w_1- \mathbb{V}|w_2|^{p-2}w_2, ((w_1 - w_2)^+) \rangle dx} \le 0.
\end{align*} 
With notation $\Omega' = \Omega \cap \{\ w_1 \ge w_2\}$, it implies from~\eqref{eq:A2-DOP} that
\begin{align*}
\int_{\Omega'} \mathbf{\Phi}(\nabla w_1, \nabla w_2)  +  \mathbb{V} \mathbf{\Phi}(w_1, w_2) dx \le 0,
\end{align*}
which gives us the following estimate from  Lemma~\ref{lem:Phi(v,v)} for every $\varepsilon>0$  that
\begin{align*}
\int_{\Omega} \mathbb{E} ((w_1 - w_2)^+) dx = \int_{\Omega'} \mathbb{E}(w_1 -w_2) dx  \le \varepsilon \int_{\Omega'} \mathbb{E}(w_1) dx.
\end{align*}
By sending $\varepsilon$ to $0$, one concludes that $w_1 \le w_2$ a.e. in $\Omega$ since $w_1 \le w_2$ a.e. on $\partial \Omega$.
\end{proof}

\begin{lemma}\label{lem:A1}
Let $u \in \mathbb{K}$ be a weak solution to~\eqref{eq:DOP} and $B$ be an open ball in $\Omega$. Assume that $v_1 \in u + W_0^{1,p}(B)$ and $v_1 \ge \psi_1$ a.e. in $B$, is the unique solution to the following obstacle problem
\begin{align}\nonumber
\int_{B} & \langle \mathcal{A}(x,\nabla v_1), \nabla v_1 - \nabla \varphi \rangle dx  + {  \int_{B} \langle \mathbb{V}|v_1|^{p-2}v_1, v_1 - \varphi\rangle dx} \\ \label{eq:DOP-1}
& \hspace{2cm} \le \int_{B} \langle \mathcal{A}(x,\nabla \psi_2), \nabla v_1 - \nabla \varphi \rangle dx + {  \int_{B} \langle \mathbb{V}|\psi_2|^{p-2}\psi_2, v_1 - \varphi\rangle dx} ,
\end{align}
for all $\varphi \in u + W_0^{1,p}(B)$ and $\varphi \ge \psi_1$ a.e in $B$. Then one can find $C>0$ such that
\begin{align}\label{est:DOP-2}
\int_{B} \mathbb{E}(v_1) dx \le C \int_{B} \left(\mathbb{E}(u) + \mathbb{E}(\psi_2)\right) dx.
\end{align}
Moreover, $v_1 \le \psi_2$ a.e in $B$.
\end{lemma}
\begin{proof}
By taking $u$ as the test function in~\eqref{eq:DOP-1}, one has
\begin{align*}
& \int_{B} \langle \mathcal{A}(x,\nabla v_1), \nabla v_1 \rangle dx  + {  \int_{B} \langle \mathbb{V}|u|^{p-2}u -\mathbb{V}|v_1|^{p-2}v_1, u - v_1\rangle dx} \\
& \hspace{2cm}  \le \int_{B} \langle \mathcal{A}(x,\nabla v_1), \nabla u \rangle dx   + \int_{B} \langle \mathcal{A}(x,\nabla \psi_2), \nabla v_1 - \nabla u \rangle dx \\
& \hspace{4cm} + {  \int_{B} \langle \mathbb{V}|u|^{p-2}u, u - v_1\rangle dx} + {  \int_{B} \langle \mathbb{V}|\psi_2|^{p-2}\psi_2, u - v_1\rangle dx},
\end{align*}
which with~\eqref{eq:A1-DOP} and~\eqref{eq:A2-DOP} implies to
\begin{align*}
& \int_{B} |\nabla v_1|^{p} dx + {  \int_{B} \mathbb{V} \mathbf{\Phi}(u,v_1) dx}  \le C \left(\int_{B} |\nabla v_1|^{p-1} |\nabla u| dx  + \int_{B} |\nabla \psi_2|^{p-1}  |\nabla v_1| dx \right. \\
& \hspace{2cm} \left. + \int_{B} |\nabla \psi_2|^{p-1}  |\nabla u| dx + {  \int_{B}  \mathbb{V}|u|^{p-1} |u-v_1| dx} + {  \int_{B}  \mathbb{V}|\psi_2|^{p-1} |u-v_1| dx} \right).
\end{align*}
It is not difficult to show~\eqref{est:DOP-2} by applying H{\"o}lder's and Young's inequalities. \\

In order to prove $v_1 \le \psi_2$ a.e. in $B$, we perform a similar method as in the proof of previous lemma. In particular, we consider $v_1 - (v_1-\psi_2)^{+}$ as the test function in~\eqref{eq:DOP-1}, combining with~\eqref{eq:A2-DOP} and Lemma~\ref{lem:Phi(v,v)} for every $\varepsilon>0$, to have
\begin{align} \label{est:DOP-4}
\int_{B} \mathbb{E}((v_1- \psi_2)^+) dx = \int_{D} \mathbb{E}(v_1- \psi_2) dx & \le \varepsilon C \int_{B \cap \{v_1 \ge \psi_2\}} \mathbb{E}(u) + \mathbb{E}(\psi_2) dx.
\end{align}
Passing $\varepsilon$ to $0$ in~\eqref{est:DOP-4}, one concludes that $v_1 \le \psi_2$ almost everywhere in $B$. 
\end{proof}

\begin{lemma}\label{lem:u-v-DOP}
Let $u \in \mathbb{K}$ be a weak solution to~\eqref{eq:DOP} and $B$ be an open ball in $\Omega$. Assume that $v \in u + W^{1,p}_0(B)$ solves the following equations
\begin{align}\label{eq:v-local}
\begin{cases} -\mathrm{div} \left(\mathcal{A}(x,\nabla v)\right) + { \mathbb{V}|v|^{p-2}v} & = 0,  \ \mbox{ in } \ B, \\ \hspace{2.5cm} v & = u, \ \mbox{ on } \ \partial B. \end{cases}
\end{align}
Then for every $\varepsilon \in (0,1)$, one can find $C = C(p,\varepsilon)>0$ such that
\begin{align}\label{est:u-v-DOP}
\fint_{B} \mathbb{E}(u - v) dx \le \varepsilon  \fint_{B} \mathbb{E}(u) dx + C \fint_{B} |\mathcal{F}|^p dx.
\end{align}
\end{lemma}
\begin{proof}
Let $v_1 \in u + W_0^{1,p}(B)$ be the unique solution to obstacle problem~\eqref{eq:DOP-1}. Thanks to Lemma~\ref{lem:A1}, we may extend $v_1$ to $\Omega \setminus B$ by $u$ so that $v_1 \in \mathbb{K}$ and $v_1-u = 0$ in $\Omega \setminus B$. Adding two inequality corresponding to the ones by taking $v_1$ and $u$ as test functions of problems~\eqref{eq:DOP} and~\eqref{eq:DOP-1} respectively, one has
\begin{align}\nonumber
\int_{B} & \left\langle \mathcal{A}(x,\nabla u) - \mathcal{A}(x,\nabla v_1), \nabla u - \nabla v_1 \right\rangle dx  + {  \int_{B} \langle \mathbb{V}(|u|^{p-2}u - |v_1|^{p-2}v_1), u - v_1\rangle dx} \\  \nonumber
& \hspace{3cm}   \le { \int_{B} f(u - v_1) dx}  + {  \int_{B} \langle \mathbb{V}|\psi_2|^{p-2}\psi_2, v_1 - u\rangle dx} \\ \label{est:DOP-5}
& \hspace{5cm}  +  \int_{B} \left\langle \mathcal{B}(x,\mathbf{F}) - \mathcal{A}(x,\nabla \psi_2), \nabla u - \nabla v_1 \right\rangle dx.
\end{align}
Combining~\eqref{est:DOP-5} with assumptions~\eqref{eq:A1-DOP} and \eqref{eq:A2-DOP} on nonlinear operators $\mathcal{A}$, $\mathcal{B}$, it follows that
\begin{align}\nonumber 
 \int_{B} \mathbf{\Phi}(\nabla u,\nabla v_1)   & + \mathbb{V} \mathbf{\Phi}(u, v_1)  dx  \le C \left({ \int_{B} |f||u - v_1| dx} + {  \int_{B} \mathbb{V}|\psi_2|^{p-1} |v_1 - u| dx} \right. \\ \label{est:DOP-6} 
 &  \hspace{1cm} \left.  +  \int_{B}  |\nabla u - \nabla v_1| |\mathbf{F}|^{p-1} dx  + \int_{B} |\nabla \psi_2|^{p-1} |\nabla u - \nabla v_1| dx\right).
\end{align}
The first term on the right hand-side of~\eqref{est:DOP-6} can be estimated by combining
 H{\"o}lder's, Sobolev's and Young's inequalities as follows
\begin{align}\nonumber
{  C \int_{B}  |f||u - v_1| dx} & \le  C \left(\int_{B} |u-v_1|^{p}dx\right)^{\frac{1}{p}} \left(\int_{B} |f|^{p'} dx\right)^{\frac{1}{p'}} \\ \nonumber
& \le C \left(\int_{B} |\nabla u - \nabla v_1|^{p} dx\right)^{\frac{1}{p}} \left(\int_{B} |f|^{p'} dx\right)^{\frac{1}{p'}}   \\ \label{ineq:S_2}
& \le \frac{\varepsilon_1}{4}  \int_{B}  |\nabla u - \nabla v_1|^p  dx +  C(\varepsilon_1) \int_{B} |f|^{p'} dx,
\end{align} 
for every $\varepsilon_1 >0$. Similar considerations apply to the remain terms on the right hand side of~\eqref{est:DOP-6}. It follows from~\eqref{ineq:S_2} that
\begin{align}\nonumber
 \int_{B} \mathbf{\Phi}(\nabla u,\nabla v_1) + \mathbb{V} \mathbf{\Phi}(u, v_1)  dx  & \le \varepsilon_1 \int_{B} \mathbb{E}(u-v_1) dx  \\ \label{est:DOP-7}
 & \hspace{1cm} + C(\varepsilon_1) \int_{B} \left(|f|^{p'} + |\mathbf{F}|^p + \mathbb{E}(\psi_2)\right) dx.
\end{align}
Thanks to Lemma~\ref{lem:Phi(v,v)}, for any $\varepsilon \in (0,1)$, there holds
\begin{align}\nonumber
\int_{B} \mathbb{E}(u-v_1) dx  &\le  \frac{\varepsilon}{2}  \int_{B} \mathbb{E}(u) dx  + C(\varepsilon) \int_{B} ( \mathbf{\Phi}(\nabla u, \nabla v_1) + {  \mathbb{V} \mathbf{\Phi}(u, v_1)}) dx,
\end{align}
which implies to the following estimate by substituting~\eqref{est:DOP-7} into the last term
\begin{align}\nonumber
\int_{B} \mathbb{E}(u-v_1) dx &\le \frac{\varepsilon}{2}  \int_{B} \mathbb{E}(u) dx  + \varepsilon_1 C(\varepsilon) \int_{B} \mathbb{E}(u-v_1) dx \\ \label{est:DOP-8a}
& \hspace{4cm}  + C(\varepsilon_1,\varepsilon) \int_{B} \left(|f|^{p'} + |\mathbf{F}|^p +  \mathbb{E}(\psi_2)\right) dx.
\end{align}
Let us chose $\varepsilon_1   = [{2C(\varepsilon)}]^{-1}$ in~\eqref{est:DOP-8a} to get that 
\begin{align}\label{est:DOP-8} 
\int_{B} \mathbb{E}(u-v_1) dx \le  \varepsilon  \int_{B} \mathbb{E}(u) dx + C(\varepsilon) \int_{B} \left(|f|^{p'} + |\mathbf{F}|^p + \mathbb{E}(\psi_2)\right) dx.
\end{align}
We consider $v_2$ as the unique weak solution to the equations below
\begin{align}\label{eq:omega-2}
\begin{cases} -\mathrm{div} \left(\mathcal{A}(x,v_2)\right) + {  \mathbb{V}|v_2|^{p-2}v_2} & = \ -\mathrm{div} \left(\mathcal{A}(x,\psi_1)\right)  + {  \mathbb{V}|\psi_1|^{p-2}\psi_1},  \ \mbox{ in } \ B,\\ 
\hspace{2cm} v_2 & = \qquad v_1, \hspace{4cm}  \ \mbox{ on } \partial B. \end{cases}
\end{align}
Since $v_2 = v_1 \ge \psi_1$ almost everywhere on $\partial B$ so it deduces from Lemma~\ref{lem:weak-max} that $v_2 \ge \psi_1$ almost everywhere in $B$. Therefore we may take $v_2$ as the test function in~\eqref{eq:DOP-1} and choose $v_1-v_2$ as the test function in variational formula solving equation~\eqref{eq:omega-2}, to observe that
\begin{align}\nonumber
& \int_{B} \left\langle \mathcal{A}(x,\nabla v_1) - \mathcal{A}(x,\nabla v_2), \nabla v_1 - \nabla v_2\right\rangle dx  + {  \int_{B} \langle \mathbb{V}|v_1|^{p-2}v_1 - \mathbb{V}|v_2|^{p-2}v_2, v_1 - v_2\rangle dx}  \\ \nonumber
& \hspace{2cm} = \int_{B}  \left\langle \mathcal{A}(x,\nabla \psi_2), \nabla v_1 - \nabla v_2\right\rangle dx - \left\langle \mathcal{A}(x,\nabla \psi_1), \nabla v_1 - \nabla v_2\right\rangle dx \\ \label{est:DOP-9}
& \hspace{3cm} + {  \int_{B} \langle \mathbb{V}|\psi_2|^{p-2}\psi_2, v_1 - v_2\rangle dx}- {  \int_{B} \langle \mathbb{V}|\psi_1|^{p-2}\psi_1, v_1 - v_2\rangle dx}.
\end{align}
From~\eqref{est:DOP-9}, the proof is essentially the same as the previous one in~\eqref{est:DOP-8}, once again we may show that
\begin{align}\label{est:DOP-10}
& \int_{B} \mathbb{E}(v_1 - v_2) dx  \le \varepsilon  \int_{B} \mathbb{E}(v_1) dx + C(\varepsilon) \int_{B} \left(\mathbb{E}(\psi_1) + \mathbb{E}(\psi_2)\right) dx,
\end{align}
for every $\varepsilon \in (0,1)$. Let $v$ be the unique weak solution to the following equations
\begin{align}\label{eq:omega-3}
\begin{cases} -\mathrm{div} \left(\mathcal{A}(x,v)\right) + {  \mathbb{V}|v|^{p-2}v} & = \ 0,  \ \ \, \mbox{ in } \ B, \\ \hspace{2.5cm} v & = \ v_2,  \ \mbox{ on } \partial B. \end{cases}
\end{align}
We emphasize that since $v_2 = v_1 = u$ on $\partial B$, so two problems~\eqref{eq:omega-3} and~\eqref{eq:v-local} are the same. A similar proof remains valid to obtain the following estimate
\begin{align}\label{est:DOP-11}
\int_{B} \mathbb{E}(v_2 - v) dx \le \varepsilon  \int_{B} \mathbb{E}(v) dx + C(\varepsilon) \int_{B} \mathbb{E}(\psi_1) dx.
\end{align}
Collecting the estimates in~\eqref{est:DOP-2}, \eqref{est:DOP-8}, \eqref{est:DOP-10} and~\eqref{est:DOP-11} to discover that
\begin{align*}
\int_{B} \mathbb{E}(u-v) dx &\le 3^{p-1}\varepsilon \int_{B} \left(\mathbb{E}(u) + \mathbb{E}(v)\right)dx  + C(\varepsilon) \int_{B} |\mathcal{F}|^p dx,
\end{align*}
which guarantees~\eqref{eq:lem-Phi(u,v)}, by noting that 
$\mathbb{E}(v) \le 2^{p-1}(\mathbb{E}(u) + \mathbb{E}(u-v))$, and changing a suitable value of $\varepsilon>0$. 
\end{proof}
\begin{remark}\label{cor:global}
We remark that the proof of~\eqref{est:u-v-DOP} in Lemma~\ref{lem:u-v-DOP} even holds for the case $B = \Omega$ which implies that $v = 0$. Therefore, the inequality~\eqref{est:u-v-DOP} becomes to the following global estimate
\begin{align}\label{eq-cor:global}
\int_{\Omega} \mathbb{E}(u) dx \le C \int_{\Omega} |\mathcal{F}|^p dx.
\end{align}
\end{remark}

Next, let us recall a reverse H{\"o}lder's inequality on $\mathbb{E}(v)$. We refer to~\cite{LO_schro} for the proof.

\begin{lemma}\label{lem:RH}
Under assumptions of Lemma~\ref{lem:u-v-DOP} and let $\mathbb{V} \in \mathcal{RH}^{\gamma}$ for some $\gamma \in [\frac{n}{p}, n)$ with $p \in (1,n)$, satisfying $\|\mathbb{V}\|_{L^{\gamma;p\gamma}(\Omega)} \le 1$. Then there exists $\delta>0$ such that if problem~\eqref{eq:DOP} satisfies the assumption $(r_0,\delta)-(\mathbb{H})$ for some $r_0>0$ then
\begin{align}\label{est:RH}
\left(\fint_{B_{\varrho}} [\mathbb{E}(v)]^{\gamma} dx \right)^{\frac{1}{\gamma}} \le  C\fint_{B_{2\varrho}} \mathbb{E}(v) dx,
\end{align}
holds for all $\varrho>0$ such that $B_{2\varrho} \subset B$.
\end{lemma}

\begin{remark}\label{remk}
Under additional assumption on the boundary that $\Omega$ is a $(r_0,\delta)$-Reifenberg flat domain, Lemma~\ref{lem:u-v-DOP} and Lemma~\ref{lem:RH} even hold near the boundary. This means the same conclusions of these lemmas can be obtained by replacing $B$ by $B_{\varrho}(\xi) \cap \Omega$ for all $\xi \in \partial \Omega$ and $0< \varrho \le \mathrm{diam}(\Omega)$. For the proof, we use the same technique as in~\cite{MPTNsub}.
\end{remark}

\section{Proofs of main results}\label{sec:proof}

The rest of this article is devoted to the proofs of main theorems. However, before dealing with Theorem \ref{theo:main-A} and \ref{theo:main-norm}, let us reproduce a standard result that plays a key ingredient in our proofs. This is a corollary of Calder\'on-Zygmund-Krylov-Safonov decomposition, that allows us to work with balls. We address the reader to~\cite[Lemma 4.2]{CC1995} for the detailed proof.

\begin{lemma}\label{lem:cover-lem}
Let a $(r_0,\delta)$-Reifenberg flat domain $\Omega$ for some $r_0>0$. Consider two measurable subsets $\mathcal{Q}\subset \mathcal{N}$ of $\Omega$ satisfying $\mathcal{L}^n\left(\mathcal{Q}\right) \le \varepsilon \mathcal{L}^n\left(B_{r_0}\right)$ for some $\varepsilon \in (0,1)$. Suppose that for any $0 < \varrho \le r_0$ and $\xi \in \Omega$, we have $B_{\varrho}(\xi) \cap \Omega \subset \mathcal{N}$ if provided $\mathcal{L}^n\left(\mathcal{Q} \cap B_{\varrho}(\xi)\right) > \varepsilon \mathcal{L}^n\left(B_{\varrho}(\xi)\right)$. Then there is $C>0$ such that $\mathcal{L}^n\left(\mathcal{Q}\right)\leq C \varepsilon \mathcal{L}^n\left(\mathcal{N}\right)$.
\end{lemma} 

\begin{proof}[Proof of Theorem~\ref{theo:main-A}]
The main purpose of this proof is to find $\varepsilon_0 \in (0,1)$ and $\delta>0$ such that~\eqref{eq:mainlambda-P} is valid for all $\lambda>0$ and $\varepsilon \in (0,\varepsilon_0)$, under assumption $(r_0,\delta)-(\mathbb{H})$ of problem~\eqref{eq:DOP} for some $r_0>0$. The inequality~\eqref{eq:mainlambda-P} can be rewritten as 
$$\mathcal{L}^n(\mathcal{Q}^{\lambda}_{\varepsilon}) \le C \varepsilon \mathcal{L}^n(\mathcal{N}^{\lambda}),$$ 
where two measurable sets $\mathcal{Q}^{\lambda}_{\varepsilon}$ and $\mathcal{N}^{\lambda}$ in this inequality are defined by 
\begin{align}\nonumber 
& \mathcal{Q}^{\lambda}_{\varepsilon} = \{{\mathbf{M}}(\mathbb{E}(u))>\varepsilon^{-\theta}\lambda, {\mathbf{M}}(|\mathcal{F}|^p) \le \beta \lambda \} \ \mbox{ and } \ \mathcal{N}^{\lambda} = \{ {\mathbf{M}}(\mathbb{E}(u))> \lambda \},
\end{align}
where $\beta = \beta(\texttt{data},\theta,\varepsilon) > 0$ will be determined later. The proof is obtained as an application of the covering Lemma~\ref{lem:cover-lem} for $\mathcal{Q}^{\lambda}_{\varepsilon}$ and $\mathcal{N}^{\lambda}$. So we only need to check all hypotheses of Lemma~\ref{lem:cover-lem}.

We may assume $\mathcal{Q}^{\lambda}_{\varepsilon} \neq \emptyset$ which gives us a point $\xi_1 \in \Omega$ satisfying ${\mathbf{M}}(|\mathcal{F}|^p)(\xi_1) \le {\beta}\lambda$. For $r_0>0$, we first show that $\mathcal{L}^n\left(\mathcal{Q}^{\lambda}_{\varepsilon}\right) \le \varepsilon \mathcal{L}^n\left(B_{R}\right)$. Using the fundamental boundedness property of ${\mathbf{M}}$ and applying estimate~\eqref{eq-cor:global} in Remark~\ref{cor:global}, there holds
\begin{align}\label{est-3.1-P}
\mathcal{L}^n(\mathcal{Q}^{\lambda}_{\varepsilon}) \le \mathcal{L}^n(\{ {\mathbf{M}}(\mathbb{E}(u))>\varepsilon^{-\theta}\lambda \}) \le \frac{C}{\varepsilon^{-\theta} \lambda}\int_{\Omega}{\mathbb{E}(u) dx} \le \frac{C}{\varepsilon^{-\theta} \lambda}\int_{\Omega}{|\mathcal{F}|^p dx}  .
\end{align}
Now we can bound the integral in the last term of~\eqref{est-3.1-P} by the integral over the ball centered at $\xi_1$ and radius $\mathrm{diam}(\Omega)$, it gives us
\begin{align*}
\mathcal{L}^n(\mathcal{Q}^{\lambda}_{\varepsilon}) 
\le \frac{C (\mathrm{diam}(\Omega))^n}{\varepsilon^{-\theta} \lambda}  {\mathbf{M}}(|\mathcal{F}|^p)(\xi_1) \le C \beta\varepsilon^{\theta}  (\mathrm{diam}(\Omega))^n \le \varepsilon  \mathcal{L}^n(B_{r_0}).
\end{align*}
We emphasize that the last inequality holds for $\beta>0$ satisfying 
$$C \beta \varepsilon^{\theta-1} \left(\mathrm{diam}(\Omega)/r_0\right)^{n} <  1.$$

In the second step, we prove the following statement by contradiction: for every $0< \varrho \le r_0$ and $\xi \in \Omega$, if $\mathcal{L}^n\left(\mathcal{Q}^{\lambda}_{\varepsilon} \cap B_{\varrho}(\xi)\right) \ge \varepsilon \mathcal{L}^n\left(B_{\varrho}(\xi)\right)$ then $B_{\varrho}(\xi) \cap \Omega \subset \mathcal{N}^{\lambda}$. Suppose that we can find $\xi_2 \in B_{\varrho}(\xi)\cap \Omega \cap (\mathbb{R}^n\setminus \mathcal{N}^{\lambda})$ and $\xi_3 \in \mathcal{Q}^{\lambda}_{\varepsilon} \cap B_{\varrho}(\xi)$ which imply to
\begin{align}\label{eq:x2-P}
{\mathbf{M}}(\mathbb{E}(u))(\xi_2) \le \lambda \ \mbox{ and } \
{\mathbf{M}}(|\mathcal{F}|^p)(\xi_3) \le \beta \lambda.
\end{align}
Therefore, in the rest of the proof, we only need to show that
\begin{align}\label{est:goal}
\mathcal{L}^n(\mathcal{Q}^{\lambda}_{\varepsilon} \cap B_{\varrho}(\xi)) < \varepsilon \mathcal{L}^n(B_{\varrho}(\xi)).
\end{align} 
For any $\zeta \in B_{\varrho}(\xi)$, it is obvious to see that $B_{r}(\zeta) \subset B_{3r}(\xi_2), \ \forall r \ge \varrho$, which allows us taking into account~\eqref{eq:x2-P} to find
\begin{align*}
\sup_{r \ge \varrho}{\fint_{B_{r}(\zeta)}{\mathbb{E}(u) dx}} \le 3^n\sup_{r \ge \varrho}{\fint_{B_{3r}(\xi_2)}{\mathbb{E}(u) dx}} \le 3^n \lambda.
\end{align*}
From this fact and the definition of the cutoff maximal function below
\begin{align*}
{\mathbf{M}}^{\varrho}(\mathbb{E}(u))(\zeta) = \sup_{0< r < \varrho}{\fint_{B_{r}(\zeta)}{\mathbb{E}(u) dx}},
\end{align*}
we may conclude that 
$${\mathbf{M}}(\mathbb{E}(u))(\zeta)  \le \max \left\{ {\mathbf{M}}^{\varrho}(\mathbb{E}(u))(\zeta) ; \  3^n \lambda \right\}, \ \mbox{ for all } \ \zeta \in B_{\varrho}(\xi).$$ For this reason, for every $\varepsilon \in (0,3^{-\frac{n}{\theta}})$, there holds
\begin{align*}
\mathcal{Q}^{\lambda}_{\varepsilon} \cap B_{\varrho}(\xi) = \{{\mathbf{M}}^{\varrho}(\mathbb{E}(u))> \varepsilon^{-\theta}\lambda; \ {\mathbf{M}}(|\mathcal{F}|^p) \le \beta \lambda \} \cap B_{\varrho}(\xi),
\end{align*}
which deduces to
\begin{align}\label{eq:res11-P}
\mathcal{Q}^{\lambda}_{\varepsilon} \cap B_{\varrho}(\xi) \subset \{{\mathbf{M}^{\varrho}}(\mathbb{E}(u))> \varepsilon^{-\theta}\lambda \} \cap B_{\varrho}(\xi).
\end{align} 
There are two cases $B_{4\varrho}(\xi) \subset\Omega$ or $B_{4\varrho}(\xi) \cap \partial\Omega \neq \emptyset$. Let us set $B_{j} = B_{2^{j}\varrho}(\xi)$ with $j = 1,2$ in the first case. Otherwise there exists $\xi_4 \in \partial \Omega$ such that $|\xi_4 - \xi| = \mathrm{dist}(\xi,\partial \Omega) \le 4\varrho$, we choose $B_{j} = B_{2^{j+2}\varrho}(\xi_4) \cap \Omega$ with $j = 1,2$ in this case. Let $v \in u + W_0^{1,p}(B_2)$ be the unique solution to the following problem
\begin{equation}\nonumber 
\begin{cases} -\mbox{div} \left( \mathcal{A}(x,\nabla v)\right) + \mathbb{V}|v|^{p-2}v & = \ 0, \quad  \quad \mbox{ in } B_2,\\ 
\hspace{2.2cm} v & = \ u, \qquad \mbox{ on } \partial B_2.\end{cases}
\end{equation}
Lemma~\ref{lem:u-v-DOP} and Lemma~\ref{lem:RH} combining with Remark~\eqref{remk} give us the following reverse H{\"o}lder's inequality
\begin{align}\label{est:101-P}
\left(\fint_{B_1} \mathbb{E}(v)^{\gamma} dx\right)^{\frac{1}{\gamma}} \le  C\fint_{B_2}\mathbb{E}(v)dx,
\end{align}
and the comparison estimate as below
\begin{align}\label{est:100-P}
\fint_{B_2} \mathbb{E}(u-v) dx \le \delta \fint_{B_2} \mathbb{E}(u) dx + C(\delta) \fint_{B_2} |\mathcal{F}|^p dx,
\end{align}
for all $\delta \in (0,1)$. By the definition of $B_2$, one can find $k \in \mathbb{N}$ such that $B_2 \subset B_{k\varrho}(\xi_2) \cap B_{k\varrho}(\xi_3)$. Moreover, it notes that $\mathcal{L}^n(B_2) \sim \mathcal{L}^n(B_{k\varrho}(\xi_2)) \sim \varrho^n$. One obtains from~\eqref{eq:x2-P} that
\begin{align}\label{eq:com-2-P}
\fint_{B_2}{\mathbb{E}(u)dx} \le \frac{\mathcal{L}^n(B_{k\varrho}(\xi_2))}{\mathcal{L}^n(B_2)} \fint_{B_{k\varrho}(\xi_2)}{\mathbb{E}(u)dx} \le C {\mathbf{M}}(\mathbb{E}(u))(\xi_2) \le C \lambda,
\end{align}
and similarly
\begin{align}\label{eq:com-22-P}
\fint_{B_2}{|\mathcal{F}|^pdx} \le \frac{\mathcal{L}^n(B_{k\varrho}(\xi_3))}{\mathcal{L}^n(B_2)} \fint_{B_{k\varrho}(\xi_3)}{|\mathcal{F}|^pdx} \le C {\mathbf{M}}(|\mathcal{F}|^p)(\xi_3) \le C \beta \lambda.
\end{align}
Substituting~\eqref{eq:com-2-P} and~\eqref{eq:com-22-P} into~\eqref{est:100-P}, we find
\begin{align}\label{eq:u-v-B3-P}
\fint_{B_2} \mathbb{E}(u-v) dx \le C (\delta + C(\delta) \beta)\lambda.
\end{align}
By the definition of $B_1$ we can check that $B_{r}(\zeta) \subset B_{2\varrho}(\xi) \subset B_1$ for all $r \in (0,\varrho)$. Thus we deduce from~\eqref{eq:res11-P} with an elementary inequality that
\begin{align}\nonumber
\mathcal{L}^n(\mathcal{Q}^{\lambda}_{\varepsilon} \cap B_{\varrho}(\xi)) & \le \mathcal{L}^n(\{{\mathbf{M}}^{\varrho}(\chi_{B_1}\mathbb{E}(u-v))> 2^{-p}\varepsilon^{-\theta}\lambda \} \cap B_{\varrho}(\xi))\\  \label{eq:estV-1-P}
& \qquad + \mathcal{L}^n(\{{\mathbf{M}}^{\varrho}(\chi_{B_1}\mathbb{E}(v))> 2^{-p}\varepsilon^{-\theta}\lambda \} \cap B_{\varrho}(\xi)) =: \mathrm{I} + \mathrm{II}.
\end{align}
To estimate the first term $\mathrm{I}$, we apply Lemma~\ref{lem:bound-M} for $s = 1$ and~\eqref{eq:u-v-B3-P} to arrive
\begin{align}\label{est:I-P}
\mathrm{I} &\le \frac{C}{\varepsilon^{-\theta}\lambda} \int_{B_1} \mathbb{E}(u-v) dx \le \frac{C \mathcal{L}^n(B_2)}{\varepsilon^{-\theta}\lambda} \fint_{B_2} \mathbb{E}(u-v) dx \le C (\delta\varepsilon^{\theta} + C(\delta)\beta\varepsilon^{\theta})\varrho^n.
\end{align}
We now apply Lemma~\ref{lem:bound-M} with $s = \gamma>1$ and the reverse H{\"o}lder inequality~\eqref{est:101-P} to have
\begin{align}\label{est:II-P}
\mathrm{II} &\le \frac{C \mathcal{L}^n(B_1)}{\left(\varepsilon^{-\theta}\lambda\right)^{\gamma}} \fint_{B_1} \mathbb{E}(v)^{\gamma} dx \le \frac{C \mathcal{L}^n(B_1)}{\left(\varepsilon^{-\theta}\lambda\right)^{\gamma}} \left( \fint_{B_2} \mathbb{E}(v) dx \right)^{\gamma} \le C \varepsilon^{\theta \gamma} (1 + C(\delta) \beta)^{\gamma} \varrho^n.
\end{align}
Here, the last estimate comes from~\eqref{eq:com-2-P} and~\eqref{eq:u-v-B3-P} as below
\begin{align*}
\fint_{B_2} \mathbb{E}(v) dx \le 2^{p-1} \left(\fint_{B_2} \mathbb{E}(u) dx + \fint_{B_2} \mathbb{E}(u-v) dx \right) \le C (1 + \delta + C(\delta) \beta)\lambda,
\end{align*}
Substituting~\eqref{est:I-P} and~\eqref{est:II-P} into~\eqref{eq:estV-1-P}, there holds
\begin{align*}
\mathcal{L}^n(\mathcal{Q}^{\lambda}_{\varepsilon} \cap B_{\varrho}(\xi)) &\le C \left[\delta\varepsilon^{\theta} + C(\delta)\beta\varepsilon^{\theta} + \varepsilon^{\theta \gamma} (1 + \delta + C(\delta) \beta)^{\gamma} \right] \mathcal{L}^n(B_{\varrho}(\xi)).
\end{align*}
We emphasize that for every $\delta \in (0,1)$ one may choose $\beta \le \delta C^{-1}(\delta)$ to observes 
\begin{align}\label{est:IV-P}
\mathcal{L}^n(\mathcal{Q}^{\lambda}_{\varepsilon} \cap B_{\varrho}(\xi)) &\le C (\delta\varepsilon^{\theta} + \varepsilon^{\theta \gamma}) \mathcal{L}^n(B_{\varrho}(\xi)).
\end{align}
Since $\theta \gamma > 1$, it is possible to choose $\delta$ in~\eqref{est:IV-P} satisfying $C (\delta \varepsilon^{\theta} + \varepsilon^{\theta \gamma}) < \varepsilon$, to get~\eqref{est:goal} and thus complete the proof. 
\end{proof}

\bigskip

\begin{proof}[Proof of Theorem~\ref{theo:main-norm}]
For every $0< q < \gamma$ and $0<s < \infty$, let us choose $\theta$ such that
$\frac{1}{\gamma} < \theta < \min\left\{1,\frac{1}{q}\right\}$.
Thanks to Theorem~\ref{theo:main-A}, there exist $\delta>0$, $\varepsilon_0 \in (0,1)$, $\beta > 0$ such that if problem~\eqref{eq:DOP} satisfies the assumption $(r_0,\delta)-(\mathbb{H})$ for some $r_0>0$, then the following inequality
\begin{align} \label{est:A-1}
& \mathcal{L}^n(\{{\mathbf{M}}(\mathbb{E}(u))>\varepsilon^{-\theta}\lambda\}) \leq C \varepsilon \mathcal{L}^n (\{ {\mathbf{M}}(\mathbb{E}(u))> \lambda\}) +  \mathcal{L}^n(\{{\mathbf{M}}(|\mathcal{F}|^p) > \beta \lambda\}),
\end{align}
holds  for all $\varepsilon \in (0,\varepsilon_0)$ and $\lambda>0$. By changing of variables and substituting~\eqref{est:A-1} into the integral in the norm formula of $\mathbf{M}(\mathbb{E}(u))$, one obtains that
\begin{align}\nonumber
\|\mathbf{M}(\mathbb{E}(u))\|^s_{L^{q,s}(\Omega)}  & \le C q \varepsilon^{-s\theta}\int_0^\infty{\lambda^s \left[\varepsilon \mathcal{L}^n(\{{\mathbf{M}}(\mathbb{E}(u))>\lambda\})\right]^{\frac{s}{q}}\frac{d\lambda}{\lambda}} \\ \nonumber
& \hspace{3cm} +  q \varepsilon^{-s\theta}\int_0^\infty{\lambda^s \left[\mathcal{L}^n(\{{\mathbf{M}}(|\mathcal{F}|^p) > \beta \lambda\})\right]^{\frac{s}{q}}\frac{d\lambda}{\lambda}}  \\ \label{est:5b}
& \le  C \varepsilon^{s(\frac{1}{q}-\theta)} \|\mathbf{M}(\mathbb{E}(u))\|^s_{L^{q,s}(\Omega)} +   \beta^{s} \varepsilon^{-s\theta} \|\mathbf{M}(|\mathcal{F}|^p)\|^s_{L^{q,s}(\Omega)}.
\end{align}
From the setting of $\theta$, we have $\frac{1}{q}-\theta > 0$ which ensure that the choice $\varepsilon \in (0,\varepsilon_0)$ in~\eqref{est:5b} satisfying $C \varepsilon^{s(\frac{1}{q}-\theta)} \le \frac{1}{2}$ is possible, one obtains~\eqref{eq:main-A}. The similar way can be done in the other case $s = \infty$ to finish the proof.  
\end{proof}

\bigskip

For the proof of the last theorem, one may perform the similar technique as in the proofs of Theorem~\ref{theo:main-A} and Theorem~\ref{theo:main-norm}. The difference only lies on the application of the boundedness of fractional maximal function $\mathbf{M}_{\alpha}$ instead of Hardy-Littlewood maximal function $\mathbf{M}$. We sketch here the important and different points. For more details we refer the reader to~\cite{PNJDE,MPTNsub} for the original idea and~\cite{PNmix} for the shorter one.\\

\begin{proof}[Proof of Theorem~\ref{theo:main-M-alpha}]
We first prove that for every $\alpha \in [0,\frac{n}{\gamma})$, $\theta \in (\frac{n-\alpha\gamma}{n\gamma},1)$, one can find some constants $\varepsilon_0 = \varepsilon_0(\texttt{data},\alpha,\theta) \in (0,1)$, $\delta = \delta(\texttt{data},\alpha,\theta,\varepsilon)>0$ and $\beta = \beta(\texttt{data},\alpha,\theta,\varepsilon) > 0$ such that if problem~\eqref{eq:DOP} satisfies the assumption $(r_0,\delta)-(\mathbb{H})$ for some $r_0>0$, then the following estimate
\begin{align}\label{eq:lambda-C} 
& \mathcal{L}^n(\{{\mathbf{M}_{\alpha}}(\mathbb{E}(u))>\varepsilon^{-\theta}\lambda, {\mathbf{M}_{\alpha}}(|\mathcal{F}|^p) \le \beta \lambda\}) \leq C \varepsilon \mathcal{L}^n(\{ {\mathbf{M}_{\alpha}}(\mathbb{E}(u))> \lambda\}),
\end{align}
holds for any $\varepsilon \in (0,\varepsilon_0)$ and $\lambda>0$. The main idea is also applying Lemma~\ref{lem:cover-lem} for two following sets
\begin{align}\nonumber 
& \mathcal{Q}^{\lambda}_{\alpha,\varepsilon} = \{{\mathbf{M}_{\alpha}}(\mathbb{E}(u))>\varepsilon^{-\theta}\lambda, {\mathbf{M}_{\alpha}}(|\mathcal{F}|^p) \le \beta \lambda \} \ \mbox{ and } \ \mathcal{N}^{\lambda}_{\alpha} = \{ {\mathbf{M}_{\alpha}}(\mathbb{E}(u))> \lambda \}.
\end{align}
To show the first assumption of Lemma~\ref{lem:cover-lem}, we prove an inequality as in~\eqref{est-3.1-P}. Applying Lemma~\ref{lem:bound-M} with $s=1$, there holds
\begin{align}\nonumber
\mathcal{L}^n(\mathcal{Q}^{\lambda}_{\alpha,\varepsilon}) & \le C \left(\frac{1}{\varepsilon^{-\theta} \lambda}\int_{\Omega}{\mathbb{E}(u) dx}\right)^{\frac{n}{n-\alpha}} \le C \left(\frac{1}{\varepsilon^{-\theta} \lambda}\int_{\Omega}{|\mathcal{F}|^p dx} \right)^{\frac{n}{n-\alpha}} \\ \nonumber
& \le C \left(\frac{1}{\varepsilon^{-\theta} \lambda} (\mathrm{diam}(\Omega))^{n -\alpha} \mathbf{M}_{\alpha}(|\mathcal{F}|^p)(\xi_1) \right)^{\frac{n}{n-\alpha}} \\ \label{est:C-1}
& \le C \left(\beta \varepsilon^{\theta}\right)^{\frac{n}{n-\alpha}} (\mathrm{diam}(\Omega)/r_0)^n \mathcal{L}^n (B_{r_0}) \le C \varepsilon \mathcal{L}^n (B_{r_0}).
\end{align}
In the last inequality of~\eqref{est:C-1}, we choose $\beta>0$ such that $\left(\beta \varepsilon^{\theta}\right)^{\frac{n}{n-\alpha}} (\mathrm{diam}(\Omega)/r_0)^n < \varepsilon$.

The second assumption in Lemma~\ref{lem:cover-lem} will be also proved by contradiction. In this way, we also the cutoff fractional maximal function to obtain the following estimate
\begin{align}\nonumber
\mathcal{L}^n(\mathcal{Q}^{\lambda}_{\alpha,\varepsilon} \cap B_{\varrho}(\xi)) & \le \mathcal{L}^n(\{{\mathbf{M}}^{\varrho}_{\alpha}(\chi_{B_1}\mathbb{E}(u-v))> 2^{-p}\varepsilon^{-\theta}\lambda \} \cap B_{\varrho}(\xi))\\  \label{eq:C-2}
& \hspace{3cm} + \mathcal{L}^n(\{{\mathbf{M}}^{\varrho}_{\alpha}(\chi_{B_1}\mathbb{E}(v))> 2^{-p}\varepsilon^{-\theta}\lambda \} \cap B_{\varrho}(\xi)),
\end{align}
for any $\varepsilon$ small enough and $B_1$ defined as the proof of Theorem~\ref{theo:main-A}. Next we will applying Lemma~\ref{lem:bound-M} with different values of $s$ to bound two terms in~\eqref{eq:C-2} respectively. A slight change in the proof actually shows that
\begin{align}\nonumber
\mathcal{L}^n(\mathcal{Q}^{\lambda}_{\alpha,\varepsilon} \cap B_{\varrho}(\xi)) &\le C \left[ \left(\delta\varepsilon^{\theta} + C(\delta)\beta\varepsilon^{\theta}\right)^{\frac{n}{n-\alpha}} + \left(\varepsilon^{\theta \gamma} (1 + C(\delta) \beta)^{\gamma}\right)^{\frac{n}{n-\alpha \gamma}} \right] \mathcal{L}^n(B_{\varrho}(\xi)).
\end{align}
The most important remark here is that $\frac{\theta \gamma n}{n-\alpha \gamma} > 1$ which allows us to choose $\delta$ and $\beta$ depending on $\varepsilon$ such that
\begin{align*}
\left[ \left(\delta\varepsilon^{\theta} + C(\delta)\beta\varepsilon^{\theta}\right)^{\frac{n}{n-\alpha}} + \left(\varepsilon^{\theta \gamma} (1 + C(\delta) \beta)^{\gamma}\right)^{\frac{n}{n-\alpha \gamma}} \right] < \varepsilon.
\end{align*}
Hence we may conclude~\eqref{eq:lambda-C} which implies to~\eqref{eq:main-C} by the same computation as in the proof of Theorem~\ref{theo:main-norm}. This finises the proof.
\end{proof}


\begin{thebibliography}{99}

%

\bibitem{AM2007}  E. Acerbi, G. Mingione, {\em Gradient estimates for a class of parabolic systems}, Duke Math. J. {\bf 136} (2007), 285--320.

\bibitem{AMS2001} A. Ambrosetti, A. Malchiod, S. Secchi, {\em Multiplicity results for some nonlinear Schr\"odinger equations with potentials}, Arch. Ration. Mech. Anal., {\bf 159}(2001), 253--271.


\bibitem{Baroni2014} P. Baroni, {\em Lorentz estimates for obstacle parabolic problems}, Nonlinear Anal. {\bf 96} (2014) 167--188.


\bibitem{BS1991} F. A. Berezin, M. A. Shubin, {\em The Schr\"odinger Equation}, Mathematics and Its Applications (Soviet Series), vol. 66, Kluwer Academic Publishers Group, Dordrecht, 1991.

\bibitem{BFM2001} M. Bildhauer, M. Fuchs, G. Mingione, {\em A priori gradient bounds and local $C^{1,\alpha} $ estimates for (double) obstacle problems under non-standard growth conditions}, Z. Anal. Anwendungen {\bf 20} (2001), no. 4, 959-985.

\bibitem{BBHV} M. Bramanti, L. Brandolini, E. Harboure, B. Viviani, {\em Global $W^{2,p}$ estimates for nondivergence elliptic operators with potentials satisfying a reverse H\"older condition}, Ann. Mat. Pura Appl. {\bf 191}(2012), 339-362.

\bibitem{BDM2011} V. B\"ogelein, F. Duzzar, G. Mingione, {\em Degenerate problems with irregular obstacles}, J. Reine Angew. Math. {\bf 650} (2011), 107-160.


\bibitem{BS2012} V. B\"ogelein, C. Scheven, {\em Higher integrability in parabolic obstacle problems}, Forum Math., {\bf 24}(5)(2012), 931-972.

\bibitem{BCW2012} S.-S. Byun, Y. Cho, L. Wang, {\em Calder\'on-Zygmund theory for nonlinear elliptic problems with irregular obstacles}, J. Funct. Anal. {\bf 263} (10) (2012) 3117-3143.

\bibitem{BLO2020} S.-S. Byun, S. Liang, J. Ok, {\em Irregular Double Obstacle Problems with Orlicz Growth}, J. Geom. Anal. {\bf 30}(2020), 1965-1984.


\bibitem{BR2020} S.-S. Byun, S. Ryu, {\em Gradient estimates for nonlinear elliptic double obstacle problems}, Nonlinear Anal. {\bf 194} (2020), 111333.

\bibitem{BW2} S.-S. Byun, L. Wang, {\em Elliptic equations with BMO nonlinearity in Reifenberg domains},  Adv. Math. {\bf 219}(6)  (2008), 1937-1971.



\bibitem{CP1998} L. A. Caffarelli, I. Peral, {\it On $W^{1,p}$ estimates for elliptic equations in divergence form},  Commun. Pure Appl. Math. {\bf 51}(1) (1998), 1-21.

\bibitem{CC1995} L. A. Caffarelli, X. Cabr\'e, {\em Fully nonlinear elliptic equations}, American Mathematical Society Colloquium Publications, American Mathematical Society, Providence {\bf 43}(1) (1995), 1-21.
 


\bibitem{Choe1991} H. J. Choe, {\em A regularity theory for a general class of quasilinear elliptic partial differential equations and obstacle problems}, Arch. Rational Mech. Anal. {\bf 114} (1991), 383-394.

\bibitem{Choe1992} H. J. Choe, {\em Regularity for certain degenerate elliptic double obstacle problems}, J. Math. Anal. Appl., {\bf 169}(1)(1992), 111-126.

\bibitem{CL1991} H. J. Choe, J. L. Lewis, {\em On the obstacle problem for quasilinear elliptic equations of $p$-Laplacian type}, SIAM J. Math. Anal. {\bf 22} (1991), no. 3, 623-638.

\bibitem{Choe2016} H. J. Choe,  P. Souksomvang, {\em Elliptic gradient constraint problem}, Comm. in Partial Differential Equations, {\bf 41}(12) (2016), 1918-1933.

\bibitem{CoMi2016} M. Colombo, G. Mingione, {\it Calder\'on-Zygmund estimates ans non-uniformly elliptic operators},  J. Funct. Anal. {\bf 136}(4) (2016), 1416-1478.

\bibitem{MMV1989} G. Dal Maso, U. Mosco, M. A. Vivaldi, {\em A pointwise regularity theory for the two-obstacle problem}, Acta Math. {\bf 163} (1–2)(1989), 57-107.

\bibitem{DK2011} H. Dong, D. Kim, {\em On the $L^p$ solvability of higher order parabolic and elliptic systems with BMO coefficients}, Arch. Rational Mech. Anal., {\bf 199}(2011), no.3, 880-941.

\bibitem{Duzamin2} F. Duzaar, G. Mingione, {\em Gradient estimates via linear and nonlinear potentials}, J. Funct. Anal. {\bf 259} (2010), 2961-2998.



\bibitem{Eleuteri2007} M. Eleuteri, {\em Regularity results for a class of obstacle problems}, Appl Math {\bf 52}, 137-170 (2007).


\bibitem{EH2011} M. Eleuteri, J. Habermann, {\em A H\"older continuity result for a class of obstacle problems under non standard growth conditions}, Math. Nachr. 284 (11–12) (2011), 1404-1434.

\bibitem{EHL2013} M. Eleuteri, P. Harjulehto, T. Lukkari, {\em Global regularity and stability of solutions to obstacle problems with nonstandard growth}, Rev. Mat. Complut. {\bf 26}(2013), 147-181.

\bibitem{Fichera} G. Fichera, {\em Problemi elastostatici con vincoli unilaterali: il problema di Signorini con ambigue condizioni al contorno}, Atti Accad. Naz. Lincei Mem. Cl. Sci. Fis. Mat. Natur. Sez. la {\bf 7} (8) (1963-64) 91-140. 


\bibitem{Friedman} A. Friedman, {\em Variational Principles and Free-Boundary Problems}, in: Wiley–Interscience Publication, Pure Appl. Math., John Wiley and Sons, Inc., New York, 1982.



\bibitem{KZ1991} T. Kilpel\"ainen, W.P. Ziemer, {\em Pointwise regularity of solutions to nonlinear double obstacle problems}, Ark. Mat. {\bf 29} (1)(1991), 83-106.

\bibitem{KS1980} D. Kinderlehrer, G. Stampacchia, {\em An Introduction to Variational Inequalities and Their Applications}, Pure Appl. Math., vol. 88, Academic Press, New York, London, 1980.

\bibitem{Krylov} N. V. Krylov, {\em Parabolic and elliptic equations with VMO coefficients}, Comm. Partial Differential Equations, {\bf 32} (1-3) (2007), 453–475.

\bibitem{LO_schro} M. Lee, J. Ok, {\em Interior and boundary $W^{1,q}$-estimates for elliptic quasilinear equations of Schr\"odinger type}, J. Differ. Equ. (2020). https://doi.org/10.1016/j.jde.2020.03.028.

\bibitem{Liu} Y. Liu, J. Dong, {\em Some estimates of higher order Riesz transform related to Schr\"odinger operator}, Potential. Anal. {\bf 32} (2010), 41-55.

\bibitem{Lieberman1991} G. M. Lieberman, {\em Regularity of solutions to some degenerate double obstacle problems}, Indiana Univ. Math. J. {\bf 40} (3)(1991), 1009-1028.

\bibitem{MP12} T. Mengesha, N. C. Phuc, {\it Global estimates for quasilinear elliptic equations on Reifenberg flat domains},  Arch. Ration. Mech. Anal. {\bf 203}(1)  (2012), 189-216.


\bibitem{MZ1991} J. H. Michael, W. P. Ziemer, {\em Existence of solutions to obstacle problems}, Nonlinear Anal. {\bf 17}(1)(1991), 45-71.

\bibitem{Mi3} G. Mingione, {\em Gradient estimates below the duality exponent}, Math. Ann. {\bf 346} (2010), 571-627.

\bibitem{55QH4} Q.-H. Nguyen, N. C. Phuc, {\em  Good-$\lambda$ and Muckenhoupt-Wheeden type bounds, with applications to quasilinear elliptic equations with gradient power source terms and measure data}, Math. Ann. {\bf 374}(1-2) (2019), 67-98.

\bibitem{PNmix} T.-N. Nguyen, M.-P. Tran, {\em Lorentz improving estimates for the $p$-Laplace equations with mixed data}, Nonlinear Anal. (to appear), arXiv:2003.04530.

\bibitem{PN_dist} T.-N. Nguyen, M.-P. Tran, {\em Level-set inequalities on fractional maximal distribution functions and applications to regularity theory}, arXiv:2004.06394.



\bibitem{RT2011} J. F. Rodrigues, R. Teymurazyan, {\em On the two obstacles problem in Orlicz-Sobolev spaces and applications}, Complex Var. Elliptic Equ. {\bf 56} (7-9) (2011), 769-787.

\bibitem{Rodfrigues1987} J. F. Rodrigues, {\em Obstacle Problems in Mathematical Physics}, North Holland, Amsterdam (1987).


\bibitem{Scheven1} C. Scheven, {\em Elliptic obstacle problems with measure data: potentials and low order regularity}, Publ. Mat. {\bf 56} (2) (2012), 327-374.

\bibitem{Scheven2} C. Scheven, {\em Gradient potential estimates in non-linear elliptic obstacle problems with measure data}, J. Funct. Anal. {\bf 262} (6)(2012), 2777-2832.

\bibitem{Shen2} Z. Shen, {\em On the Neumann problem for Schr\"odinger operators in Lipschitz domains}, Indiana Univ. Math. J. {\bf 43} (1994), no. 1, 143-176.

\bibitem{Shen} Z. Shen, {\em $L^p$ estimates for Schr\"odinger operators with certain potentials}, Ann. Inst. Fourier {\bf 45} (1995), no. 2, 513-546.

\bibitem{Stampacchia} G. Stampacchia, {\em Formes bilineaires coercitives sur les ensembles convexes}, C. R. Acad. Sci. Paris, Ser. I {\bf 258}(1964) 4413-4416. 

\bibitem{Sugano} S. Sugano, {\em $L^p$ estimates for some Schr\"odinger type operators and a Calder\'on-Zygmund operator of Schr\"odinger type}, Tokyo J. Math. {\bf 30} (2007), 179-197.

\bibitem{MPT2018} M.-P. Tran, {\em Good-$\lambda$ type bounds of quasilinear elliptic equations for the singular case},  Nonlinear Anal. {\bf 178} (2019), 266-281.

\bibitem{PNCRM} M.-P. Tran, T.-N. Nguyen, {\em Generalized good-$\lambda$ techniques and applications to weighted Lorentz regularity for quasilinear elliptic equations}, C. R. Acad. Sci. Paris, Ser. I {\bf 357}(8) (2019), 664-670.

\bibitem{MPTNsub} M.-P. Tran, T.-N. Nguyen, {\em Weighted Lorentz gradient and point-wise estimates for solutions to quasilinear divergence form elliptic equations with an application}, arXiv:1907.01434.

\bibitem{PNJDE} M.-P. Tran, T.-N. Nguyen, {\em New gradient estimates for solutions to quasilinear divergence form elliptic equations with general Dirichlet boundary data}, J. Differ. Equ. {\bf 268}(4) (2020), 1427-1462.

\bibitem{PNnonuniform} M.-P. Tran, T.-N. Nguyen, {\em Global Lorentz estimates for non-uniformly nonlinear elliptic equations via fractional maximal operators}, J. Math. Anal. Appl. (2020). https://doi.org/10.1016/j.jmaa.2020.124084.

\bibitem{Troianiello} G. M. Troianiello, {\em Elliptic Differential Equations and Obstacle Problems}, The University Series in Mathematics. Plenum Press, New York, xiv+353 pp. ISBN: 0-306-42448-7, (1987).

\end{thebibliography}
\end{document}